\documentclass[12pt]{article}
\usepackage{amsmath,amsfonts}
\usepackage{verbatim}
\usepackage{latexsym}
\usepackage{graphicx}
\usepackage{color}
\usepackage{epstopdf}

\makeatletter \@addtoreset{equation}{section}

\begin{document}

\newcommand{\E}{\mathbb{E}}
\newcommand{\PP}{\mathbb{P}}
\newcommand{\RR}{\mathbb{R}}

\newcommand{\Dt}{\D t}
\newcommand{\bX}{\bar X}
\newcommand{\bx}{\bar x}
\newcommand{\by}{\bar y}
\newcommand{\bp}{\bar p}
\newcommand{\bq}{\bar q}

\newtheorem{theorem}{Theorem}[section]
\newtheorem{lemma}[theorem]{Lemma}
\newtheorem{coro}[theorem]{Corollary}
\newtheorem{defn}[theorem]{Definition}
\newtheorem{assp}[theorem]{Assumption}
\newtheorem{expl}[theorem]{Example}
\newtheorem{prop}[theorem]{Proposition}
\newtheorem{rmk}[theorem]{Remark}

\newcommand\tq{{\scriptstyle{3\over 4 }\scriptstyle}}
\newcommand\qua{{\scriptstyle{1\over 4 }\scriptstyle}}
\newcommand\hf{{\textstyle{1\over 2 }\displaystyle}}
\newcommand\athird{{\scriptstyle{1\over 3 }\scriptstyle}}
\newcommand\hhf{{\scriptstyle{1\over 2 }\scriptstyle}}

\newcommand{\eproof}{\indent\vrule height6pt width4pt depth1pt\hfil\par\medbreak}

\def\a{\alpha} \def\g{\gamma}
\def\e{\varepsilon} \def\z{\zeta} \def\y{\eta} \def\o{\theta}
\def\vo{\vartheta} \def\k{\kappa} \def\l{\lambda} \def\m{\mu} \def\n{\nu}
\def\x{\xi}  \def\r{\rho} \def\s{\sigma}
\def\p{\phi} \def\f{\varphi}   \def\w{\omega}
\def\q{\surd} \def\i{\bot} \def\h{\forall} \def\j{\emptyset}

\def\be{\beta} \def\de{\delta} \def\up{\upsilon} \def\eq{\equiv}
\def\ve{\vee} \def\we{\wedge}

\def\F{{\cal F}}
\def\T{\tau} \def\G{\Gamma}  \def\D{\Delta} \def\O{\Theta} \def\L{\Lambda}
\def\X{\Xi} \def\S{\Sigma} \def\W{\Omega}
\def\M{\partial} \def\N{\nabla} \def\Ex{\exists} \def\K{\times}
\def\V{\bigvee} \def\U{\bigwedge}

\def\1{\oslash} \def\2{\oplus} \def\3{\otimes} \def\4{\ominus}
\def\5{\circ} \def\6{\odot} \def\7{\backslash} \def\8{\infty}
\def\9{\bigcap} \def\0{\bigcup} \def\+{\pm} \def\-{\mp}
\def\[{\langle} \def\]{\rangle}

\def\proof{\noindent{\it Proof. }}
\def\tl{\tilde}
\def\trace{\hbox{\rm trace}}
\def\diag{\hbox{\rm diag}}
\def\for{\quad\hbox{for }}
\def\refer{\hangindent=0.3in\hangafter=1}

\newcommand\wD{\widehat{\D}}

\thispagestyle{empty}

\title{
\bf The Truncated Milstein Method for Stochastic Differential
 Equations}

\author{
{\bf Qian Guo${}^1$\thanks{E-mail: qguo@shnu.edu.cn}, Wei Liu${}^1$\thanks{Corresponding author. E-mail: weiliu@shnu.edu.cn, lwbvb@hotmail.com},
Xuerong Mao${}^2$\thanks{E-mail: x.mao@strath.ac.uk},
Rongxian Yue${}^1$\thanks{E-mail: yue2@shnu.edu.cn}}
\\
${}^1$ Department of Mathematics, \\
Shanghai Normal University, Shanghai, China. \\
${}^2$ Department of Mathematics and Statistics, \\
University of Strathclyde, Glasgow G1 1XH, U.K.
}
\date{}

\maketitle

\begin{abstract}
Inspired by the truncated Euler-Maruyama method developed in Mao (J. Comput. Appl. Math. 2015), we propose the truncated Milstein method in this paper. The strong convergence rate is proved to be close to 1 for a class of highly non-linear stochastic differential equations. Numerical examples are given to illustrate the theoretical results.


\medskip \noindent
{\small\bf MAS Classification (2000):}  65C20

\end{abstract}

\section{Introduction}

Stochastic differential equation (SDE), as a power tool to model uncertainties, has been broadly applied to many areas \cite{ALL2007a,Oks2003a,M01}. However, apart from linear SDEs, explicit solutions to most non-linear SDEs can hardly be found. Therefore, numerical approximations to SDEs become essential in the applications of SDE models.
\par
When the drift and diffusion coefficients of SDEs satisfy the global Lipschitz condition, different kinds of numerical approximates have been broadly studied. We refer the readers to the monographs \cite{K01,Mil02,Pla2010} for the detailed introductions and discussions.
\par
Due to the simple structure and easy to programme, explicit methods, such as the Euler-Maruyama method, have been widely used \cite{Higham2011}. But when the global Lipschitz condition is disturbed, the classical Euler-Maruyama method has been proved divergent \cite{Hut01,HJK2013}.
\par
One of the natural candidates to tackle the divergence caused by the non-linearities in coefficients is implicit method. Many works have been devoted to implicit methods \cite{Mil01,Hu01,H01,Szp11,MS2013a,NS2014,YHJ2014,WGW2012,DNS2014,BT2004,ACR2009,Schurz2012}. Despite the good performance of the strong convergence, implicit methods have their own disadvantage that some non-linear systems need to be solved in each iteration, which may be computationally expensive and introduce some more errors.
\par
Another way to tackle SDEs with non-global Lipschitz coefficients is to modify the drift and diffusion coefficients in the numerical methods. Following this approach, one can construct explicit methods that are able to converge to SDEs with coefficients allowed to grow super-linearly. The tamed Euler method \cite{Hut02,HJ2015} is one of the most popular explicit methods that were developed particularly for the super-linear SDEs. In addition, we refer the readers to \cite{Sab13,Gan01,ZWH2014} for the simplified proofs of the strong convergence for the tamed Euler method, the tamed Milstein method and the semi-tamed method, respectively.
\par
More recently, Mao in \cite{M15} proposed a new explicit method called the truncated Euler-Maruyama method. The new method focuses on those SDEs with both the drift and diffusion coefficients allowed to grow super-linearly. In \cite{Mao2016}, Mao further proved that the strong convergence rate of the method could be arbitrarily close to a half. Mao and his collaborators also studied the asymptotic behaviour of the method in \cite{GLMY2017}.
\par
Apart from the stand-alone research interests of the strong convergence of numerical methods, the property of the strong convergence could also be used to improve the convergence rate of estimating the expectation of some random variable by using the Multi-level Monte Carlo (MLMC) method \cite{Giles2008a}. Furthermore, Giles in \cite{Giles2008b} pointed out that a numerical method with the strong convergence rate of one could better cooperate with the MLMC method.
\par
Therefore, in this paper we propose the truncated Milstein method, which is an explicit method and has the strong convergence rate of arbitrarily closing to one. In this work, both of the drift and diffusion coefficients of the SDEs under investigation could grow super-linearly.
\par
This paper is organized as follows. Notations, assumptions and the truncated Milstein method will be introduced in Section \ref{mathpre}. The proofs of the main results will be presented in Section \ref{secMain}. An example together with some ideas on further research will be presented in Section \ref{secexpl}.

\section{Mathematical Preliminaries}\label{mathpre}

Throughout this paper, unless otherwise specified, let $(\Omega , \F, \PP)$ be a complete probability space with a filtration $\left\{\F_t\right\}_{t \ge 0}$ satisfying the usual conditions (that is, it is right continuous and increasing while $\F_0$ contains all $\PP$-null sets). Let $\E$ denote the expectation corresponding to $\PP$.
If $A$ is a vector or matrix, its transpose is denoted by $A^T$.
Let $B(t) = (B^1(t), B^2(t), ..., B^m(t))^T$ be
an $m$-dimensional Brownian motion defined on the space.
If $A$ is a matrix, let $|A| = \sqrt{\trace(A^TA)}$ be its trace norm.
If $x \in \RR^d$, then $|x|$ is the Euclidean norm.
For two real numbers $a$ and $b$, set $a\ve b=\max(a,b)$
and $a\we b=\min(a,b)$. If $G$ is a set, its indicator function is denoted by $I_G$, namely $I_G(x)=1$ if $x\in G$ and $0$ otherwise.
\par
Consider a $d$-dimensional SDE
\begin{equation}\label{sde}
dx(t) = f(x(t))dt + \sum_{j = 1}^m g_j(x(t))dB^{j}(t)
\end{equation}
on $t\ge 0$ with the initial value $x(0)=x_0\in\RR^d$, where
$$
f: \RR^d \to \RR^d,
\quad
g_j: \RR^d \to \RR^{d},  ~j=1,2,...,m,
$$
and $ x(t)= (x^1(t), x^2(t), ..., x^d(t))^T$.
\par
In some of the proofs in this paper, we need the more specified notation that $f=(f_1, f_2, ..., f_d)^T$, $f_i: \RR^d \rightarrow \RR$ for $i=1,2,...,d$, and $g_j = (g_{1,j}, g_{2,j},...,g_{d,j})^T$, $g_{i,j}: \RR^d \rightarrow \RR$ for $j=1,2,...,m$.
\par
For $j_1, j_2 = 1,...,m$, define
\begin{equation}\label{Lg}
L^{j_1}g_{j_2}(x) = \sum_{l=1}^d g_{l,j_1}(x)\frac{\partial g_{j_2} (x)}{\partial x^l} .
\end{equation}
\par
For the truncated Milstein method, we need that both $f$ and $g$ have continuous second-order derivatives. In addition, the following assumptions are imposed.

\begin{assp}\label{fgpoly}
There exist constants $K_2 > 0$ and $r > 0$ such that
\begin{equation*}
|f(x) - f(y)| \ve |g_j(x) - g_j(y)| \ve \left| L^{j_1} g_{j_2}(x) -  L^{j_1} g_{j_2}(y) \right|\leq K_2 (1 + |x|^r + |y|^r) |x - y|
\end{equation*}
for all $x,y \in \RR^d$ and $j, j_1, j_2=1,2,...m$.
\end{assp}

\begin{assp}\label{KhasminskiiCond}
For every $p\ge 1$, there exists a positive constant $K_1$, dependent on $p$, such that
\begin{equation*}
\left\[ x - y, f(x) - f(y) \right\] + (2p-1) \sum_{j= 1}^m \left|g_{j}(x) - g_j(y) \right|^2 \leq K_1 |x - y|^2
\end{equation*}
for all $x,y \in \RR^d$.
\end{assp}
\par \noindent
Assumptions \ref{fgpoly} and \ref{KhasminskiiCond} guarantee that the SDE (\ref{sde}) has a unique global solution.
\par
It is not hard to derive from Assumption \ref{KhasminskiiCond} that for all $x \in \RR^d$
\begin{equation}\label{KhasminskiiCondext}
\left\[ x, f(x) \right\] + (2p-1) \sum_{j= 1}^m \left| g_{j}(x) \right|^2 \leq \alpha_1(1 + |x|^2)
\end{equation}
holds for all $p\geq 1$, where $\alpha_1$ is a positive constant dependent on $p$.
\par
Moreover, Assumption \ref{KhasminskiiCond} guarantees the boundedness of
the moments of the underlying solution \cite{M01}, namely,  there exists a positive constant $K$, dependent on $t$ and $p$, such that
 \begin{equation}\label{sdepthmoment}
\E |x(t)|^{2p} \leq K \left( 1 + |x(0)|^{2p} \right).
\end{equation}
\par \noindent
From Assumption \ref{fgpoly} we can obtain that for all $x \in \RR^d$
\begin{equation}\label{fgpolyext}
|f(x)| \ve |g_j(x)| \leq \alpha_2( 1 + |x|^{r+1}),~j=1,2,...m,
\end{equation}
where $\alpha_2$ is a positive constant.
\par \noindent
For $l= 1,2,...d$, set
\begin{equation*}
f'_l(x) = \left( \frac{\partial f_l(x)}{\partial x^{1}}, \frac{\partial f_l(x)}{\partial x^{2}}, ..., \frac{\partial f_l(x)}{\partial x^{d}} \right)~\text{and}~f''_l(x) = \left( \frac{\partial^2 f_l(x) }{\partial x^j \partial x^i} \right)_{i,j},~i,j=1,2,...,d.
\end{equation*}
And for $n= 1,2,...m$, $l = 1,2,...,d$, set
\begin{equation*}
g'_{l,n}(x) = \left( \frac{\partial g_{l,n}(x)}{\partial x^{1}}, \frac{\partial g_{l,n}(x)}{\partial x^{2}}, ..., \frac{\partial g_{l,n}(x)}{\partial x^{d}} \right)~\text{and}~g''_{l,n}(x) = \left( \frac{\partial^2 g_{l,n}(x) }{\partial x^j \partial x^i} \right)_{i,j},~i,j=1,2,...,d.
\end{equation*}
We further assume that for $n= 1,2,...m$ and $l = 1,2,...,d$, there exists a positive constant $\alpha_3$ such that
\begin{equation}\label{dfgspoly}
|f_l'(x)| \ve |f_l''(x)| \ve |g_{l,n}'(x)| \ve |g_{l,n}''(x)| \leq \alpha_3  (1 + |x|^{r+1} ).
\end{equation}

\subsection{The Classical Milstein Method}
Define a uniform mesh $\mathcal{T}^N: 0=t_0<t_1<\cdots<t_N=T$ with $t_k=k\Delta$, where $\Delta=T/N$ for $N\in \mathbb{N}$, the classical Milstein method  \cite{Mil1974} is

\begin{equation*}
y_{k+1}= y_k + f(y_k) \Delta + \sum_{j= 1}^m g_j(y_k) \Delta B^j_k + \sum_{j_1= 1}^m \sum_{j_2 = 1}^m L^{j_1}g_{j_2} (y_k) I_{j_1,j_2}^{t_k,t_{k+1}},
\end{equation*}
where
\begin{equation*}
L^{j_1} = \sum_{l=1}^d g_{l,j_1}  \frac{\partial}{\partial x^l} ~\text{and} ~ I_{j_1,j_2}^{t_k,t_{k+1}} = \int_{t_k}^{t_{k+1}} \int_{t_k}^{s_2} d B^{j_1}(s_1) d B^{j_2}(s_2).
\end{equation*}
\par \noindent
When the diffusion coefficient $g$ satisfies the commutativity condition that
\begin{equation*}
L^{j_1} g_{l,j_2} = L^{j_2} g_{l,j_1}, ~\text{for}~j_1,j_2=1,...,m ~ \text{and} ~ l = 1,...,d,
\end{equation*}
the classical Milstein method is simplified into
\begin{equation*}
y_{k+1}= y_k + f(y_k) \Delta + \sum_{j= 1}^m g_j(y_k) \Delta B^j_k + \frac{1}{2} \sum_{j_1= 1}^m \sum_{j_2 = 1}^m L^{j_1}g_{j_2} (y_k) \Delta B_k^{j_1} \Delta B_k^{j_2} - \frac{1}{2} \sum_{j= 1}^m L^{j}g_{j} (y_k) \Delta,
\end{equation*}
where the property, $I_{j_1,j_2}^{t_k,t_{k+1}} + I_{j_2,j_1}^{t_k,t_{k+1}} = \Delta B_k^{j_1} \Delta B_k^{j_2}$ for $j_1 \ne j_2$, is used.
\par
In this paper, we only consider the case of the commutative diffusion coefficient. For the case of the non-commutative diffusion coefficient, the truncated Milstein method may still be applicable. But more complicated notations and new techniques will be involved. Due to the length of the paper, we will report the more general case in the future work.

\subsection{The Truncated Milstein Method}
For $j = 1, ..., m$ and $l = 1, ..., d$, define the derivative of the vector $g_j(x)$ with respect to $x^l$ by
\begin{equation*}
G_j^{l}(x):=\frac{\partial}{\partial x^l} g_j(x) = \left( \frac{\partial g_{1,j}(x) }{\partial x^l}, \frac{\partial g_{2,j}(x) }{\partial x^l}, ...,\frac{\partial g_{d,j}(x) }{\partial x^l}\right)^T.
\end{equation*}
\par
To define the truncated Milstein method, we first choose a strictly increasing continuous function $\mu : \RR_{+} \rightarrow \RR_{+}$ such that $\mu(u) \rightarrow \infty$ as $u \rightarrow \infty$ and
\begin{equation}\label{formofmu}
\sup_{|x|\leq u} (|f(x)| \ve |g_j(x)| \ve |G_j^{l}(x)| ) \leq \mu(u)
\end{equation}
for any $u \geq 2$, $j = 1, ..., m$ and $l = 1, ..., d$.
\par
Denote the inverse function of $\mu$ by $\mu^{-1}$. We see that $\mu^{-1}$ is a strictly increasing continuous function from $[\mu(0),+\infty)$ to $\RR_{+}$. We also choose a number $\Delta^* \in (0,1]$ and a strictly decreasing function $h: (0,\Delta^*] \rightarrow (0,+\infty)$ such that
\begin{equation}\label{deltahdelta}
h(\Delta^*) \geq \mu(1),~~\lim_{\Delta \rightarrow 0}h(\Delta) = \infty ~~\text{and}~~\Delta^{1/4} h(\Delta) \leq 1,~~ \forall \Delta \in (0,\D^*].
\end{equation}
\par
For a given step size $\Delta \in (0,1)$ and any $x \in \RR^d$, define the truncated functions by
\begin{equation}\label{DefnTrunc1}
\tilde{f}(x) = f\left( (|x| \we \mu^{-1}(h(\Delta)) ) \frac{x}{|x|} \right),
\end{equation}
\begin{equation}\label{DefnTrunc2}
\tilde{g}_j(x) = g_j\left( (|x| \we \mu^{-1}(h(\Delta)) ) \frac{x}{|x|} \right), ~j=1,2,...m,
\end{equation}
and
\begin{equation}\label{DefnTrunc3}
 \tilde{G}_j^{l}(x) = G_j^{l}\left( (|x| \we \mu^{-1}(h(\Delta)) ) \frac{x}{|x|} \right),~j = 1, ..., m,~l = 1, ..., d,
\end{equation}
where we set $x/|x| = 0$  if $x = 0$.  It is not hard to see that for any $x \in \RR^d$
\begin{equation}\label{hdeltabdfgG}
|\tilde{f}(x)| \ve |\tilde{g}_j(x)| \ve |\tilde{G}_j^{l}(x)| \leq \mu(\mu^{-1}(h(\Delta))) = h(\Delta).
\end{equation}
That is to say, all the truncated functions $\tilde{f}$, $\tilde{g}$ and $\tilde{G}_j^{l}$ are bounded although $f$, $g$ and $G_j^{l}$ may not. The next lemma illustrates that those truncated functions preserve (\ref{KhasminskiiCondext}) for all $\Delta \in (0, \Delta^*]$.
\begin{lemma}
Assume that (\ref{KhasminskiiCondext}) holds.  Then, for all $\Delta \in (0, \Delta^*]$ and any $x \in \RR^d$,
	\begin{equation}\label{truncatedKhasminskiiCondext}
	\left\[ x, \tilde{f}(x) \right\] + (2p-1) \sum_{j= 1}^m \left| \tilde{g}_{j}(x) \right|^2 \leq 2\alpha_1(1 + |x|^2).
	\end{equation}
\end{lemma}
The proof of this lemma is the same as that of Lemma 2.4 in \cite{M15}, so we omit it here. We should of course point out that
it was required that $h(\Delta^*) \geq \mu(2)$ in \cite{M15}, but we observe that the proof of Lemma 2.4 in \cite{M15}  still works if
$h(\Delta^*) \geq \mu(1)$ and that is why in this paper we only impose $h(\Delta^*) \geq \mu(1)$ as stated in (\ref{deltahdelta}).

\par \noindent
The truncated Milstein method is defined by
\begin{eqnarray}\label{theTMmethod}
Y_{k+1} \nonumber &=& Y_{k} + \tilde{f}(Y_{k}) \Delta + \sum_{j=1}^m \tilde{g}_j(Y_{k}) \Delta B^j_k \nonumber \\
&&+ \frac{1}{2} \sum_{j_1= 1}^m \sum_{j_2 = 1}^m \sum_{l=1}^d \tilde{g}_{l,j_1} (Y_k) \tilde{G}_{j_2}^{l}(Y_k) \Delta B_k^{j_2} \Delta B_k^{j_1} - \frac{1}{2} \sum_{j= 1}^m \sum_{l=1}^d \tilde{g}_{l,j} (Y_k) \tilde{G}_j^{l} (Y_k) \Delta.
\end{eqnarray}
To simplify the notation, we set
\begin{equation*}
L^{j_1}\tilde{g}_{j_2} (x) :=  \sum_{l=1}^d \tilde{g}_{l,j_1} (x) \tilde{G}_{j_2}^{l}(x).
\end{equation*}
The continuous version of the truncated Milstein method is defined by
\begin{equation}\label{continuousTM}
	Y(t) = \bar{Y}(t) + \int_{t_k}^{t} \tilde{f}(\bar{Y}(s))ds + \sum_{j=1}^m \int_{t_k}^{t} \tilde{g}_j(\bar{Y}(s)) dB^j(s) + \sum_{j_1=1}^m \int_{t_k}^t \sum_{j_2=1}^m L^{j_1}\tilde{g}_{j_2} (\bar{Y}(s)) \Delta B^{j_2}(s)  d B^{j_1}(s),
\end{equation}
where $\bar{Y}(t) = Y_k$ for $t_k \leq t < t_{k+1}$ and $\Delta B^{j_2}(s) = \sum_{k=0}^{\infty} I_{\{t_k \leq s < t_{k+1} \}} (B^{j_2}(s) - B^{j_2}(t_k))$.

\subsection{Boundedness of the Moments}
It is obvious from (\ref{hdeltabdfgG}) that for any $T > 0$
\begin{equation*}
\sup_{0 \leq t \leq T} \E \left| Y(t) \right|^{2p} <\infty.
\end{equation*}
However, it is not so  clear that for any $T > 0$
\begin{equation*}
\sup_{0< \Delta \leq \Delta^*}\sup_{0 \leq t \leq T} \E \left| Y(t) \right|^{2p} <\infty.
\end{equation*}
This is what we are going to prove in this subsection. Firstly, we show that $Y(t)$ and $\bar{Y}(t)$ are close to each other.
\begin{lemma}\label{YYbar}
	For any $\Delta \in (0,\Delta^{*}]$, any $t \geq 0$ and any $p \geq 1$,
	\begin{equation*}
	\E |Y(t) - \bar{Y}(t)|^{2p} \leq c \Delta^{p}(h(\Delta))^{2p},
	\end{equation*}
	where $c$ is a positive constant independent of $\Delta$. Consequently, for any $t \geq 0$
	\begin{equation*}
	\lim_{\Delta \rightarrow 0} \E |Y(t) - \bar{Y}(t)|^{2p} = 0.
	\end{equation*}
\end{lemma}
\begin{proof}
Fix the step size $\Delta \in (0,\Delta^{*}]$ arbitrarily.
For any $t \geq 0$,  there exists a unique integer $k \geq 0$ such that $t_k \leq t < t_{k+1}$. By the elementary inequality $|\sum_{i=1}^m a_i|^{2p} \leq m^{2p-1} \sum_{i=1}^m |a_i|^{2p} $, we derive from (\ref{continuousTM})   that
\begin{eqnarray*}
\E |Y(t) - \bar{Y}(t)|^{2p} &\leq& c \E \bigg( \left| \int_{t_k}^{t} \tilde{f}(\bar{Y}(s))ds \right|^{2p} + \left| \sum_{j=1}^m \int_{t_k}^{t} \tilde{g}_j(\bar{Y}(s)) dB^j(s) \right|^{2p} \\ \nonumber
&&+ \left| \sum_{j_1=1}^m \int_{t_k}^t  \sum_{j_2=1}^m L^{j_1}\tilde{g}_{j_2} (\bar{Y}(s)) \Delta B^{j_2}(s)  d B^{j_1}(s) \right|^{2p} \bigg),
\end{eqnarray*}
where $c$ is a positive constant independent of $\Delta$ that may change from line to line.
Then by the elementary inequality, the H\"older inequality and Theorem 7.1 in \cite{M01} (Page 39), we have
\begin{eqnarray*}
\E |Y(t) - \bar{Y}(t)|^{2p}
&\leq& c \bigg(  \Delta^{2p-1} \E \int_{t_k}^{t} \left| \tilde{f}(\bar{Y}(s)) \right|^{2p} ds  + \Delta^{(2p-2)/2} \sum_{j=1}^m \E \int_{t_k}^{t} \left| \tilde{g}_j(\bar{Y}(s)) \right|^{2p} ds  \\ \nonumber
&&+  \Delta^{(2p-2)/2} \E \sum_{j_1=1}^m \sum_{j_2=1}^m  \int_{t_k}^t  \left| L^{j_1}\tilde{g}_{j_2} (\bar{Y}(s)) \right|^{2p} |\Delta B^{j_2}(s)|^{2p}  ds  \bigg).
\end{eqnarray*}
Applying (\ref{hdeltabdfgG}) and the fact that $\E |\Delta B^{j_2}(s)|^{2p} \leq c \Delta^p$ for $s\in [t_k,t_{k+1})$, we obtain
\begin{equation*}
\E |Y(t) - \bar{Y}(t)|^{2p}
\leq c \left( \Delta^{2p} (h(\Delta))^{2p} + \Delta^{p}(h(\Delta))^{2p} + \Delta^{2p} (h(\Delta))^{4p} \right).
\end{equation*}
By (\ref{deltahdelta}), we see $\Delta^{p} (h(\Delta))^{2p} \leq \Delta^{p/2}$. Therefore, the assertion holds. \eproof
\end{proof}
Now we are ready to establish the boundedness of moments of the truncated Milstein approximate solution.
\begin{lemma}\label{TMbound}
Let (\ref{KhasminskiiCondext}) hold. Then for any $\Delta \in (0, \Delta^*]$ and any $T > 0$
\begin{equation*}
\sup_{0 < \Delta \leq \Delta^*}\sup_{0 \leq t \leq T}\E |Y(t)|^{2p} \leq K \left( 1 + \E |Y(0)|^{2p} \right),
\end{equation*}
where K is a positive constant dependent on $T$ but independent of $\Delta$.
\end{lemma}
\begin{proof}
It follows from (\ref{continuousTM}) that
\begin{equation}\label{CTM}
Y(t) = {Y}(0) + \int_{0}^{t} \tilde{f}(\bar{Y}(s))ds + \sum_{j=1}^m \int_{0}^{t} \tilde{g}_j(\bar{Y}(s)) dB^j(s) + \sum_{j_1=1}^m \int_{0}^t \sum_{j_2=1}^m L^{j_1}\tilde{g}_{j_2} (\bar{Y}(s)) \Delta B^{j_2}(s)  d B^{j_1}(s).
\end{equation}
By the It\^o formula, we have
\begin{eqnarray*}
\E |Y(t)|^{2p} &\leq& \E |Y(0)|^{2p} + 2p \E \int_{0}^{t} |Y(s)|^{2p-2} \left\[ Y(s), \tilde{f}(\bar{Y}(s) \right\] ds \nonumber \\
&&+ 2p \E \int_0^{t} \frac{2p-1}{2} |Y(s)|^{2p-2} \left| \sum_{j_1 = 1}^m \tilde{g}_{j_1}(\bar{Y}(s))
+ \sum_{j_1 = 1}^m \sum_{j_2 = 1}^m L^{j_1} \tilde{g}_{j_2} (\bar{Y}(s)) \Delta B^{j_2}(s) \right|^2 ds,
\end{eqnarray*}
where the facts that $2p |Y(s)|^{2p-2} \sum_{j_1 = 1}^m \left\[ Y(s), \tilde{g}_{j_1}(\bar{Y}(s)) + \sum_{j_2 = 1}^m L^{j_1} \tilde{g}_{j_2} (\bar{Y}(s)) \Delta B^{j_2}(s) \right\]$ is $\F_s$-measurable and
\begin{equation*}
\E \left(\sum_{j_1 = 1}^m\int_0^t 2p |Y(s)|^{2p-2}  \left\[ Y(s), \tilde{g}_{j_1}(\bar{Y}(s)) + \sum_{j_2 = 1}^m L^{j_1} \tilde{g}_{j_2} (\bar{Y}(s)) \Delta B^{j_2}(s) \right\] d B^{j_1}(s)  \right) = 0
\end{equation*}
are used. We rewrite the inequality as
\begin{eqnarray*}
\E |Y(t)|^{2p} &\leq& \E |Y(0)|^{2p} + 2p \E \int_0^t |Y(s)|^{2p-2} \left( \left\[ \bar{Y}(s), \tilde{f}(\bar{Y}(s)) \right\] + \frac{2p-1}{2} \left| \sum_{j_1 = 1}^m \tilde{g}_{j_1}(\bar{Y}(s))\right|^2 \right) ds \nonumber \\
&&+ p(2p - 1) \E \int_0^t |Y(s)|^{2p-2} \left| \sum_{j_1 = 1}^m \sum_{j_2 = 1}^m L^{j_1} \tilde{g}_{j_2} (\bar{Y}(s)) \Delta B^{j_2}(s)  \right|^2 ds \nonumber \\
&&+ 2 p \E \int_0^t |Y(s)|^{2p-2}  \left\[ Y(s) - \bar{Y}(s), \tilde{f}(\bar{Y}(s)) \right\] ds.
\end{eqnarray*}
By (\ref{hdeltabdfgG}) and (\ref{truncatedKhasminskiiCondext}), we see
\begin{eqnarray*}
\E |Y(t)|^{2p} &\leq& \E |Y(0)|^{2p} + K \E \int_0^t |Y(s)|^{2p-2} \left( 1 + |\bar{Y}(s)|^2 \right) ds \nonumber \\
&&+ K \E \int_0^t |Y(s)|^{2p-2} \left| h(\Delta) \right|^4 \Delta ds + 2p \E \int_0^t |Y(s)|^{2p-2} \left\[  Y(s) - \bar{Y}(s), \tilde{f}(\bar{Y}(s)) \right\] ds,
\end{eqnarray*}
where K is a positive constant independent of $\Delta$ and it may change from line to line but its exact value has no use to our analysis.
\par
Applying the Young inequality that
\begin{equation*}
a^{2p - 2} b \leq \frac{2p - 2}{2p} a^{2p} + \frac{1}{p} b^{p},
\end{equation*}
we obtain
\begin{eqnarray}\label{Ex2p}
\E |Y(t)|^{2p} &\leq& \E |Y(0)|^{2p} + K \int_0^t \E |Y(s)|^{2p} ds + K  \int_0^t \E |\bar{Y}(s)|^{2p} + Kt \nonumber \\
&&+ K \int_0^t \left( \left| h(\Delta) \right|^4 \Delta \right)^p ds + K \E \int_0^t \left| Y(s) - \bar{Y}(s) \right|^p \left|\tilde{f}(\bar{Y}(s))\right|^p ds.
\end{eqnarray}
By Lemma \ref{YYbar}, (\ref{deltahdelta}) and (\ref{hdeltabdfgG}), we have
\begin{equation}\label{intYYbar}
 \E \int_0^t \left| Y(s) - \bar{Y}(s) \right|^p |\tilde{f}(\bar{Y}(s))|^p ds \leq c \int_0^t \Delta^{p/2} (h(\Delta))^{2p} ds \leq c t.
\end{equation}
Substituting (\ref{intYYbar}) into (\ref{Ex2p}), by using (\ref{deltahdelta}) we then get
\begin{equation*}
\E |Y(t)|^{2p} \leq \E |Y(0)|^{2p} + Kt + Kct + K \int_0^t \left( \sup_{0 \leq u \leq s} \E |Y(u)|^p \right) ds.
\end{equation*}
As the sum of the right-hand-side terms in the above inequality is an increasing function of $t$,  we have
\begin{equation*}
\sup_{0\leq s \leq t} \E |Y(s)|^{2p} \leq \E |Y(0)|^{2p} + Kt + Kct + K \int_0^t \left( \sup_{0 \leq u \leq s} \E |Y(u)|^p \right) ds.
\end{equation*}
By the Gronwall inequality, we obtain
\begin{equation*}
\sup_{0\leq s \leq T} \E |Y(s)|^{2p} \leq K \left( 1 + \E |Y(0)|^{2p} \right),
\end{equation*}
where K is a positive constant independent of $\Delta$. Therefore, the proof is complete. \eproof
\end{proof}

\section{Main Results}\label{secMain}
 If a function  $\phi:\mathbb{R}^d \rightarrow \mathbb{R}^d$ is twice differentiable, then the following Taylor formula
\begin{equation}\label{taylorformula1}
	\begin{split}
		&\phi(x)-\phi(x^*) = \phi'(x^*)(x-x^*) + R_1(\phi)
	\end{split}
\end{equation}
holds, where $R_1(\phi)$ is the remainder term
\begin{equation}\label{R1}
	\begin{split}
		R_1(\phi) =& \int_0^1(1-\varsigma) \phi''(x^*+\varsigma(x-x^*))(x-x^*,x-x^*) d\varsigma.
	\end{split}
\end{equation}
For any $x, h_1,h_2 \in \mathbb{R}^d$, the derivatives have the following expressions
\begin{equation}\label{derivative}
	\phi'(x)(h_1) = \sum_{i=1}^{d} \frac{\partial \phi}{\partial x^{i}}h_1^{i}, \quad
	\phi''(x)(h_1,h_2) = \sum_{i,j=1}^{d} \frac{\partial^2 \phi}{\partial x^{i}\partial x^{j}}h_1^{i}h_2^{j}.
\end{equation}
Here,
\begin{equation*}
\frac{\partial \phi}{\partial x^{i}} = \left(\frac{\partial \phi_1}{\partial x^{i}}, \frac{\partial \phi_2}{\partial x^{i}}, ..., \frac{\partial \phi_d}{\partial x^{i}}  \right), \quad \phi = (\phi_1, \phi_2, ..., \phi_d).
\end{equation*}
Replacing $x$ and $x^*$ in \eqref{taylorformula1} by $Y(t)$ and $\bar{Y}(t)$, respectively, from \eqref{continuousTM} we have
\begin{equation}\label{taylorformula2}
	\begin{split}
		\phi(Y(t))-\phi(\bar{Y}(t))  = \phi'(\bar{Y}(t))\big(\sum_{j=1}^m \int_{t_k}^{t} \tilde{g}_j(\bar{Y}(s)) dB^j(s))\big) + \tilde{R}_1(\phi),
	\end{split}
\end{equation}
where
\begin{equation}\label{RT1}
	\begin{split}
		\tilde{R}_1(\phi)=\phi'(\bar{Y}(t)) \Big(\int_{t_k}^{t} \tilde{f}(\bar{Y}(s))ds + \sum_{j_1=1}^m \int_{t_k}^t \sum_{j_2=1}^m L^{j_1}\tilde{g}_{j_2} (\bar{Y}(s)) \Delta B^{j_2}(s)  d B^{j_1}(s)\Big) +R_1(\phi).
	\end{split}
\end{equation}
By (\ref{Lg}) and (\ref{derivative}), we find
\begin{equation}\label{eq:sigmac}
	\tilde{g}_{i}'(x) \big(\tilde{g}_{j}(x) \big)= L^{j}\tilde{g}_{i}(x).
\end{equation}
Therefore, by \eqref{eq:sigmac},  replacing $\phi$ in \eqref{taylorformula2} by $g_{i}$ gives
\begin{equation} \label{eq:sigmaTaylor}
	\begin{split}
		\tilde{R}_1(g_{i}) = g_{i}({Y}(t))-g_{i}(\bar{Y}(t))-\sum_{j=1}^{m}L^{j}g_{i}(\bar{Y}(t))\Delta B^{j}(t)
	\end{split}
\end{equation}
for $t_k\leq t<t_{k+1}$.

We need the following lemmas to prove our main result.
\begin{lemma} \label{mblem0}
	If Assumptions \ref{fgpoly}, \ref{KhasminskiiCond} and (\ref{dfgspoly}) hold, then for all $p\geq 1$ and $j_1,j_2 = 1, ..., m$,
	\begin{equation}
		\sup_{0 < \Delta \leq \Delta^*}\sup_{0 \leq t \leq T} \left[ \mathbb{E}|f(Y(t))|^{p} \vee \mathbb{E} |f'(Y(t))|^{p} \vee \mathbb{E} |g(Y(t))|^{p} \vee \mathbb{E} |L^{j_1} g_{j_2}(Y(t))|^p\right] < \infty.
	\end{equation}
\end{lemma}
From Lemma \ref{TMbound}, the results hold immediately.

\begin{lemma} \label{mblem}
	If Assumptions \ref{fgpoly} and \ref{KhasminskiiCond} hold, then for all $p \geq 1$ and $j = 1, ..., m$,
	\begin{equation} \label{d7}
	\sup_{0 \leq t \leq T} \left[ \mathbb{E}|x(t))|^{p} \vee \mathbb{E} |f(x(t))|^{p} \vee \mathbb{E} |g_j(x(t))|^{p} \right]<\infty.
\end{equation}
\end{lemma}
The proof is similar to that of (\ref{sdepthmoment}).

\begin{lemma} \label{Rlem}
If Assumptions \ref{fgpoly}, \ref{KhasminskiiCond} and (\ref{dfgspoly}) hold, then for $i =1,2,..., m$ and all $p \geq 1$
	\begin{equation} \label{R_estimate}
		\E|\tilde{R}_1(f) |^p \vee \E|\tilde{R}_1(g_{i})|^p \leq C \Delta^p(h(\Delta))^{2p},
	\end{equation}
	where $C$ is a positive constant independent of $\Delta$.
\end{lemma}
\begin{proof}
We first give an estimate on $\E|\tilde{R}_1(f)|^p$.   Applying Lemmas  \ref{YYbar} and \ref{TMbound}, we can find a constant $C$ such that
\begin{equation}\label{R1mu}
	\begin{split}
		\E |R_1(f)|^p
		\leq &  \int_0^1 (1-\varsigma)^p~ \E \Big|f^{''}(\bar{Y}(t)+\varsigma({Y}(t)-\bar{Y}(t)))\, \left({Y}(t)-\bar{Y}(t),{Y}(t)-\bar{Y}(t)\right)\Big|^p d\varsigma
		\\ \leq &  \int_0^1 \left[\E \Big|f^{''}(\bar{Y}(t)+\varsigma({Y}(t)-\bar{Y}(t))) \Big|^{2p}\,  \E \Big|{Y}(t)-\bar{Y}(t)\Big|^{4p} \right]^{1/2} d\varsigma
		\\ \leq &   C\left(1+\E|Y(t)|^{2p(r+1)}+\E|\bar{Y}(t)|^{2p(r+1)}\right)^{1/2} \cdot\left(\E|{Y}(t)-\bar{Y}(t)|^{4p}\right)^{1/2}
		\\ \leq & C \Delta^p \left(h(\Delta)\right)^{2p},
	\end{split}
\end{equation}
where the polynomial growth condition (\ref{dfgspoly}) on $f''(x)$, the H\"{o}lder inequality and the Jensen inequality have been used.
To estimate $\E|\tilde{R}_1(f)|^p$, we derive from \eqref{RT1}  that
\begin{align}\label{R1p}	
\E|\tilde{R}_1(f)|^p &\leq C\Big[\Delta^p \E \left|f'(\bar{Y}(t))\tilde{f}(\bar{Y}(t))\right|^p \nonumber \\
	&+ \frac{1}{2}\sum^m_{j_1,j_2=1}\E \left|f'(\bar{Y}(t))\Big(L^{j_1} \tilde{g}_{j_2}(\bar{Y}(t))(\Delta B^{j_1}(t) \Delta B^{j_2}(t)-\delta_{j_1,j_2}\Delta)\Big)\right|^p+\E|{R}_1(f)|^p\Big]
\end{align}
for $t\in[t_k,t_{k+1})$, where the Kronecker delta $\delta_{j_1,j_2}$ is a piecewise function of variables $j_1$ and $j_2$.
Note that $t-t_{k} \leq \Delta$, by using the H\"{o}lder inequality and the Burkholder-Davis-Gundy inequality we have
\begin{equation}\label{eq:DeltaW2}
\begin{array}{cl}
 \E |\Delta B^{j_1}(t) \Delta B^{j_2}(t)-\delta_{j_1,j_2}\Delta|^p
 &\leq 2^{p-1} \big[\E |\Delta B^{j_1}(t) \Delta B^{j_2}(t)|^p+\Delta^p\big] \\
 &\leq 2^{p-1} \big[\E |\Delta B^{j_1}(t)|^{2p} \E |\Delta B^{j_2}(t)|^{2p}\big]^{1/2}+2^{p-1}\Delta^p \\
 &\leq 2^{p} \Delta^p.
\end{array}
\end{equation}
Using Lemma \ref{mblem0}, (\ref{hdeltabdfgG}) and the H\"{o}lder inequality, we can show that for $0\leq t \leq T, 1\leq j_1,j_2\leq m$
\begin{equation}\label{con8}
	\begin{split}
		&\E |f'(\bar{Y}(t))\tilde{f}(\bar{Y}(t))|^p \leq \big[\E|f'(\bar{Y}(t))|^{2p} \cdot \E|\tilde{f}(\bar{Y}(t))|^{2p}\big]^{1/2}
		\leq C (h(\Delta))^{p},                                                       \\
		&\E \left|f'(\bar{Y}(t))L^{j_1} \tilde{g}_{j_2}(\bar{Y}(t))\right|^p \leq \big[\E|f'(\bar{Y}(t))|^{2p} \cdot \E|L^{j_1} \tilde{g}_{j_2}(\bar{Y}(t))|^{2p}\big]^{1/2}
		\leq C (h(\Delta))^{2p}.
	\end{split}
\end{equation}
Now,
substituting (\ref{R1mu}), \eqref{eq:DeltaW2} and (\ref{con8}) into (\ref{R1p}) and making use of the  independence of $\bar{Y}(t)$ and
$\Delta B^{j_1}(t), \Delta B^{j_2}(t)$, we obtain
\begin{equation*}\label{R1pmu}
	\E|\tilde{R}_1(f) |^p  \leq C \Delta^p(h(\Delta))^{2p}
\end{equation*}
as required.  Similarly, we can show
\begin{equation*}\label{R1pf}
	\E|\tilde{R}_1(g_{i})|^p\leq C \Delta^p(h(\Delta))^{2p}.
\end{equation*}
The proof is complete. \eproof
\end{proof}
\par \noindent
For any real number $R > |x(0)|$, we define two stopping times
\begin{equation*}
\tau_R:= \inf \{t\geq 0, |x(t)| \geq R \} ~~\text{and}~~\rho_R:= \inf \{t \geq 0, |Y(t)| \geq R\}.
\end{equation*}

\begin{theorem}\label{thmthetageqT}
Let Assumptions \ref{fgpoly},  \ref{KhasminskiiCond} and condition (\ref{dfgspoly}) hold. Given any real number $R > |x_0|$,  if $\Delta \in (0, \Delta^{*}]$ is chosen to be sufficiently small such that  $\mu^{-1}(h(\Delta))\geq R$, then
	\begin{equation*}
	\E \left( \left| e(t\wedge \theta) \right|^{2p}  \right)\leq C \Delta^{2p}(h(\Delta))^{4p},
	\end{equation*}
where $\theta :=\tau_R \wedge \rho_R$ and $e(t) :=x(t)-Y(t)$.
\end{theorem}
\begin{proof} By the It\^{o} formula, we can show that for $0\leq t\leq T$,
\begin{equation}\label{eq:ito1}
	\begin{split}
	&	\E \left( \left| e(t\wedge \theta) \right|^{2p}  \right)\\
		 = & 2p \E \int_0^{t\wedge \theta}  \left| e(s) \right|^{2p-2} \Big\langle x(s) - {Y}(s),  f(x(s))-\tilde{f}(\bar{Y}(s)) \Big\rangle ds\\
&+ 2p\sum_{i=1}^m \E \int_0^{t\wedge \theta} \left| e(s) \right|^{2p-2} \frac{2p-1}{2}\Big| g_i(x(s)) - \tilde{g}_i(\bar{Y}(s))-\sum_{j=1}^m L^j  \tilde{g_i}(\bar{Y}(s)) \Delta B^j(s) \Big|^2 ds.
\end{split}
\end{equation}
When $0\leq s\leq t\wedge \theta$, we have $|\bar{Y}(s)|<R$ and  $\mu^{-1}(h(\Delta))\geq R$, which yields $|\bar{Y}(s)|<\mu^{-1}(h(\Delta))$. According to (\ref{DefnTrunc1}) and (\ref{DefnTrunc2}),  we have that
$$
\tilde{f}(\bar{Y}(s))={f}(\bar{Y}(s)) ~~\text{and} ~~ \tilde{g_i}(\bar{Y}(s))={g_i}(\bar{Y}(s))~~~\text{for} ~~0\leq s \leq t\wedge \theta.
$$
Therefore, it follows from (\ref{eq:ito1}) and \eqref{eq:sigmaTaylor} that
\begin{equation}\label{eq:ito2}
\begin{split}
&	\E \left( \left| e(t\wedge \theta) \right|^{2p}  \right)\\
= & 2p \E \int_0^{t\wedge \theta}  \left| e(s) \right|^{2p-2} \Big(\big\langle x(s) - {Y}(s),  f(x(s))-{f}({Y}(s)) \big\rangle + \big\langle x(s) - {Y}(s),  f(Y(s))-{f}(\bar{Y}(s)) \big\rangle \Big ) ds
\\ &
+ 2p\sum_{i=1}^m \E \int_0^{t\wedge \theta} \left| e(s) \right|^{2p-2} \frac{2p-1}{2}\Big| g_i(x(s)) - {g}_i({Y}(s))+\tilde{R}_1(g_i) \Big|^2 ds
\\
\leq & 2p \big(J_1+J_2+J_3\big),
\end{split}
\end{equation}
where
\begin{equation}\label{IntJ1}
\begin{split}
J_1&=\E \int_0^{t\wedge \theta}  \left| e(s) \right|^{2p-2} \Big(\big\langle x(s) - {Y}(s),  f(x(s))-{f}({Y}(s)) \big\rangle \\
 &~+ ({2p-1}) \sum_{i=1}^m \Big| g_i(x(s)) - {g}_i({Y}(s)) \Big|^2\Big) ds,
\end{split}~~
\end{equation}
\begin{equation}\label{IntJ2}
 J_2=\E \int_0^{t\wedge \theta}  \left| e(s) \right|^{2p-2}  \big\langle x(s) - {Y}(s),  f(Y(s))-{f}(\bar{Y}(s)) \big\rangle  ds,
\end{equation}
and
\begin{equation}\label{IntJ3}
J_3=({2p-1}) \sum_{i=1}^m \E \int_0^{t\wedge \theta} \left| e(s) \right|^{2p-2} \Big|\tilde{R}_1(g_i) \Big|^2 ds.~~~~~~~~~~~~~~~~~
\end{equation}
Applying Assumption \ref{KhasminskiiCond} to $J_1$, we obtain
\begin{equation}\label{Est_J1}
J_1\leq K_1 \E \int_0^{t\wedge \theta}  \left| e(s) \right|^{2p}  ds.
\end{equation}
Inserting the expression (\ref{taylorformula2}) into (\ref{IntJ2}) gives
\begin{equation}
J_2\leq \E \int_0^{t\wedge \theta}  \left| e(s) \right|^{2p-2}  \big\langle x(s) - {Y}(s),  f'(\bar{Y}(t))\big(\sum_{j=1}^m \int_{t_k}^{s} {g}_j(\bar{Y}(s_1)) dB^j(s_1)\big) + \tilde{R}_1(f) \big\rangle  ds.
\end{equation}
By the Young inequality and the H\"{o}lder inequality, we get
\begin{equation}
\begin{split}
J_2  \leq & C \E \int_0^{t\wedge \theta}  \left| e(s) \right|^{2p}  + \big| \big\langle x(s) - {Y}(s),  f'(\bar{Y}(t))\big(\sum_{j=1}^m \int_{t_k}^{s} {g}_j(\bar{Y}(s_1)) dB^j(s_1)\big) \big\rangle \big |^p \\
 &+ \big|\big\langle x(s) - {Y}(s),  \tilde{R}_1(f) \big\rangle \big|^p  ds\\
 \leq & C \E \int_0^{t\wedge \theta} \big( \left| e(s) \right|^{2p}  +  \big| \tilde{R}_1(f)  \big|^{2p}\big)  ds +C J_4,
\end{split}
\end{equation}
where
$$
J_4=\E \int_0^{t\wedge \theta}   \big| \big\langle x(s) - {Y}(s),  f'(\bar{Y}(t))\big(\sum_{j=1}^m \int_{t_k}^{s} {g}_j(\bar{Y}(s_1)) dB^j(s_1)\big) \big\rangle \big |^p ds.
$$
Following a very similar approach used for (3.35) in \cite{Gan01}, we can show
\begin{equation*}
J_4\leq C \Delta^{2p}.
\end{equation*}
Then, we have
\begin{equation}\label{Est_J2}
\begin{split}
J_2 \leq C \E \int_0^{t\wedge \theta} \big( \left| e(s) \right|^{2p}  +  \big| \tilde{R}_1(f)  \big|^{2p}\big)  ds +C \Delta^{2p}.
\end{split}
\end{equation}
Applying the Young inequality to (\ref{IntJ3}) gives
\begin{equation}\label{Est_J3}
J_3\leq C \sum_{i=1}^m \E \int_0^{t\wedge \theta}  \Big( \left| e(s) \right|^{2p}+ \big|\tilde{R}_1(g_i) \big|^{2p}  \Big) ds.
\end{equation}
Substituting (\ref{Est_J1}), (\ref{Est_J2}) and (\ref{Est_J3}) into (\ref{eq:ito2}), and then applying the Gronwall inequality and Lemma \ref{Rlem},  we obtain the desired result. \eproof
\end{proof}

\begin{lemma}\label{tauest}
	Let Assumptions \ref{fgpoly} and \ref{KhasminskiiCond} hold. For any real number $R > |x(0)|$, the estimate
	\begin{equation*}
	\PP \left( \tau_{R} \leq T \right) \leq \frac{K}{R^{2p}}
	\end{equation*}
	holds for some positive constant $K$ independent of $R$.
\end{lemma}
The proof of this lemma is similar to that of (\ref{sdepthmoment}). Briefly speaking, replacing $t$ by $\tau_R \we T$ in (\ref{sdepthmoment}) we see
\begin{equation*}
\E \left| x(\tau_R \we T)\right|^{2p} \leq K.
\end{equation*}
Then
\begin{equation*}
K \ge \E \left| x(\tau_R \we T)\right|^{2p} \geq \E \Big( \left|x(\tau_R)\right|^{2p} I_{\{\tau_R \leq T\} } \Big) = R^{2p} \PP \left(\tau_R \leq T \right),
\end{equation*}
which implies the assertion.
\begin{lemma}\label{rhoest}
Let (\ref{KhasminskiiCondext}) hold. For any real number $R > |x(0)|$ and any sufficiently small $\Delta \in (0, \Delta^*]$, the estimate
	\begin{equation*}
	 \PP \left( \rho_{R} \leq T \right) \leq \frac{K}{R^{2p}}
	\end{equation*}
holds for some positive constant $K$ independent of $R$ and $\Delta$.
\end{lemma}
The proof is similar to that of Lemma \ref{tauest}.
\par
We now present our main theorem.

\begin{theorem}\label{thmmain}
Let Assumptions \ref{fgpoly}, \ref{KhasminskiiCond} and (\ref{dfgspoly}) hold.  Furthermore, assume that for any given $p \geq 1$, there exists a $q \in (p, +\infty)$ and a $\Delta^*$ satisfying (\ref{deltahdelta}). In addition, if \begin{equation}\label{condonhd}
h(\Delta) \geq \mu \left( \left(\Delta^p (h(\Delta))^{2p} \right)^{-1/(q-p)}\right)
\end{equation}
holds for all sufficiently small $\Delta \in (0,\Delta^*]$, then for any fixed $T = N\Delta > 0$ and sufficiently small $\Delta \in (0,\Delta^*]$,
\begin{equation}\label{strongerror}
\E \left| x(T) - Y_N \right|^{2p} \leq K  \Delta^{2p}(h(\Delta))^{4p}
\end{equation}
holds, where $K$ is a positive constant independent of $\Delta$.
\end{theorem}

\begin{proof}
We separate the left hand side of \eqref{strongerror} into two parts
\begin{equation}\label{error2parts}
\E \left| x(T) - Y_N \right|^{2p} = \E \left( \left| x(T) - Y_N \right|^{2p} I_{\{\theta > T\}} \right) +  \E \left( \left| x(T) - Y_N \right|^{2p} I_{\{\theta \leq T\}} \right).
\end{equation}
Let us first consider the second term on the right hand side. Fix any $p \in [1, +\infty)$.  Using the Young inequality that
	\begin{equation*}
	a^{2p} b = \left(\delta a^{2q} \right)^{p/q} \left( \frac{b^{q/(q-p)}}{\delta^{p/(q-p)}}\right)^{(q-p)/q} \leq \frac{p \delta}{q} a^{2q} + \frac{q-p}{q\delta^{p/(q - p)}} b^{q/(q-p)}
	\end{equation*}
	for any $\delta >0$,
	we can have
	\begin{equation}\label{xmYpPtheta}
	\E \left( \left| x(T) - Y_N \right|^{2p} I_{\{\theta \leq T\}} \right) \leq \frac{p \delta}{q} \E \left( \left| x(T) - Y_N \right|^{2q}\right) + \frac{q-p}{q\delta^{p/(q - p)}} \PP \left( \theta \leq T \right).
	\end{equation}
Applying (\ref{sdepthmoment}) and Lemma \ref{TMbound}, we see
\begin{equation}\label{xmYpart}
\E \left( \left| x(T) - Y_N \right|^{2q}\right) \leq 2^{2q-1} \E \left( \left|x(T)\right|^{2q} + \left|Y_N \right|^{2q} \right) \leq C,
\end{equation}
where C is a positive constant independent of $R$ and $\Delta$.	
By  Lemmas \ref{tauest} and \ref{rhoest},  we also have
\begin{equation}\label{Pthetapart}
\PP \left( \theta \leq T \right) \leq \PP \left( \tau_{R} \leq T \right) +\PP \left( \rho_{R} \leq T \right) \leq \frac{2K}{R^{2q}}.
\end{equation}
Substituting (\ref{xmYpart}) and (\ref{Pthetapart}) into (\ref{xmYpPtheta}) yields
\begin{equation*}
\E \left( \left| x(T) - Y_N \right|^{2p} I_{\{\theta \leq T\}} \right) \leq \frac{Cp \delta}{q} + \frac{2K(q-p)}{qR^{2q}\delta^{p/(q - p)}}.
\end{equation*}
Choosing
\begin{equation*}
\delta = \Delta^{2p}(h(\Delta))^{4p}~~~\text{and}~~~R=\left(\Delta^p (h(\Delta))^{2p} \right)^{-1/(q-p)},
\end{equation*}
we have
\begin{equation} \label{3.31}
\E \left( \left| x(T) - Y_N \right|^{2p} I_{\{\theta \leq T\}} \right) \leq \Delta^{2p}(h(\Delta))^{4p} \left( \frac{Cp}{q} \ve \frac{2K(q-p)}{q} \right).
\end{equation}
Due to (\ref{condonhd}), we observe
\begin{equation*}
\mu^{-1} \left( h(\Delta) \right) \geq  \left(\Delta^p (h(\Delta))^{2p} \right)^{-1/(q-p)} = R
\end{equation*}
for any  $\Delta \in (0,\Delta^*)$.
Applying Theorem \ref{thmthetageqT} to the first term on the right hand side of (\ref{error2parts}) completes the proof.             \eproof
\end{proof}

Let us close this section by the following remark.
\begin{rmk}\label{R3.8}
In this paper, our conditions are imposed for
every $p\ge 1$ as we wish to show the strong $L^{2p}$-convergence rate for every $p\ge 1$.
However, our theory can also be applied to the case of \emph{some} $p\ge 1$.
For example, assume that the conditions in Theorem \ref{thmmain} hold for some  $\bar p \geq 1$  and (\ref{condonhd}) is replaced by that for the given $\bar p$, there exists a $\bar q \in (\bar p, +\infty)$ such that
\begin{equation*}
h(\Delta) \geq \mu \left( \left(\Delta^{\bar p} (h(\Delta))^{2\bar p} \right)^{-1/(\bar q- \bar p)}\right)
\end{equation*}
holds for all sufficiently small $\Delta \in (0,\Delta^*]$,  then our proof above shows clearly that for all sufficiently small $\Delta \in (0,\Delta^*]$ and for any fixed $T = N\Delta > 0$,
\begin{equation*}
\E \left| x(T) - Y_N \right|^{2\bar p} \leq K  \Delta^{2\bar p}(h(\Delta))^{4\bar p}.
\end{equation*}
\end{rmk}

\section{An Example and Further Discussion}\label{secexpl}

After the theoretical discussion on the truncated Milstein method, it is time to explain how to apply the method.   One may note from Section \ref{mathpre} that the choices of functions $\mu(u)$ and $h(\Delta)$ are essential in order to use the method.  The forms of these two functions are highly related to the structures of the drift and diffusion coefficients $f$ and $g$ of the SDE (\ref{sde}).
We shall illustrate the theory as well as  how to choose $\mu(u)$ and $h(\Delta)$ by the following example.

\begin{expl} {\rm
Consider the scalar SDE
\begin{equation*}
dx(t) = (x(t) - x^5(t))dt + x^2(t)dB(t),~~~t\geq0,
\end{equation*}
with the initial value $x(0)=1$.   The drift and diffusion coefficients are $f(x) =x-x^5$ and $g(x) =  x^2$, respectively. Clearly, both of them have continuous second-order derivatives and it is not hard to verify that Assumption \ref{fgpoly} and (\ref{dfgspoly}) are satisfied with $r=4$.  Moreover,  for any $x,y\in \RR$ and any $p \geq 1$,  we have
\begin{equation*}
\begin{split}
&~~~~(x - y) (f(x) - f(y)) + (2p-1) |g(x) - g(y)|^2 \\
&= (x-y)\big[(x-y)-(x-y)(x^4+x^3y+x^2y^2+xy^3+y^4)\big]+(2p-1)(x+y)^2(x-y)^2 \\
&= \big[1-(x^4+x^3y+x^2y^2+xy^3+y^4)+(2p-1)(x+y)^2\big]|x-y|^2.
\end{split}
\end{equation*}
But
\begin{align*}
-(x^3y+xy^3) = -xy (x^2+y^2) \le 0.5(x^2+y^2)^2 =0.5(x^4+x^4) + x^2y^2.
\end{align*}
Hence
\begin{align*}
& (x - y) (f(x) - f(y)) + (2p-1) |g(x) - g(y)|^2 \\
\le & \big[1-0.5(x^4+y^4)+2(2p-1)(x^2+y^2)\big]|x-y|^2  \\
\le & \big[1+ 4(2p-1)^2\big]|x-y|^2.
\end{align*}
That is to say, Assumption \ref{KhasminskiiCond} is fulfilled.
\par
It is clear to see
\begin{equation*}
\sup_{|x|\leq u} \left( |f(x)| \ve |g(x)| \ve |g'(x)| \right) \leq u^5, ~~\forall u \geq 2.
\end{equation*}
So we choose $\mu(u) = u^5$. Then its inverse function is $\mu^{-1}(u) = u^{1/5}$. For $\epsilon \in (0, 1/4]$, we define $h(\Delta) = \Delta^{-\epsilon}$ for $\Delta > 0$.  Due to the requirement $h(\Delta^*) \geq \mu(1)$, we can choose $\D^*=1$. Hence, (\ref{deltahdelta}) is satisfied.  To make (\ref{condonhd}) to hold, we need
\begin{equation*}
\Delta^{-\epsilon} \geq \left( \left( \Delta^{p-2p\epsilon} \right)^{-1/(q-p)} \right)^5
\end{equation*}
to hold for each $p \ge 1$. That is to say, we require
\begin{equation*}
\left( \frac{10p}{q-p} + 1 \right)\epsilon \ge \frac{5p}{q-p}.
\end{equation*}
In fact, for any given $p \ge 1$ and any small $\epsilon > 0$, we can always choose sufficiently large $q$ to make the inequality above to hold. Therefore, by Theorem \ref{thmmain} we can  conclude
\begin{equation} \label{4.1}
\E \left| x(T) - Y(T) \right|^{2p} \leq K\Delta^{2p(1-\epsilon)}, \quad \forall \D \in (0, 1].
\end{equation}
That is, the strong $L^{2p}$-convergence rate is close to $2p$ (or $L^1$-convergence rate is close to 1).

In the computer simulations, we choose $\e=0.1$ and regard the numerical solution with the step size of $2^{-16}$ as the true solution. In Figure 1, we plot the strong errors (i.e., in $L^1$) of the truncated Milstein method with step sizes $2^{-13}$, $2^{-12}$, $2^{-11}$ and $2^{-10}$, respectively.
\begin{figure}
  \centering
  \includegraphics[scale=0.8]{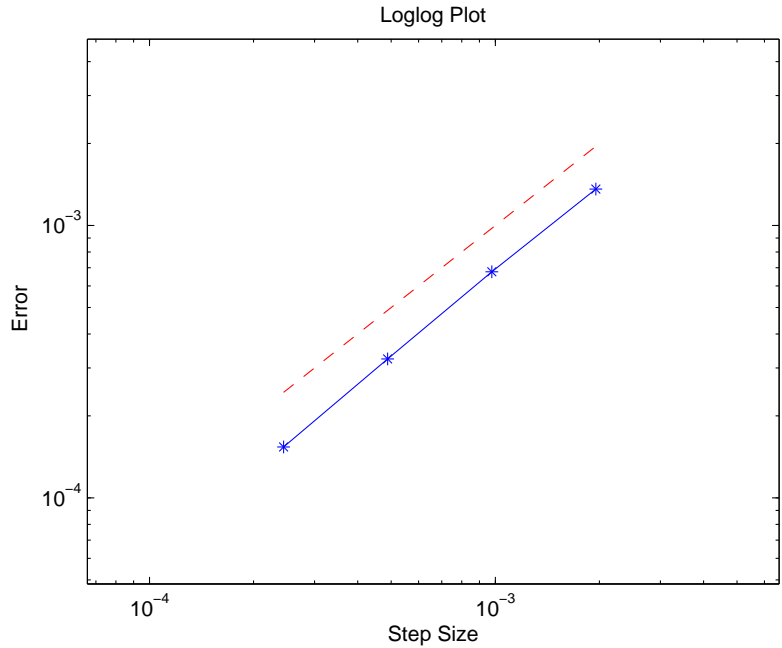}
  \label{strplt}
   \caption{The strong convergence order at the terminal time $T=2$. The red dashed line is the reference line with the slope of 1.}
\end{figure}
\par
It is interesting to observe from Figure 1 that the strong convergence rate is quite close to one, although we choose $\e=0.1$ and
the theoretical result (\ref{4.1}) only shows the rate of $0.9$.   This observation indicates that our theoretical result is somehow conservative.
}\end{expl}
We also observe from Theorem \ref{thmmain} that the strong convergence rate is highly dependent on the choices of the functions,
$\mu(\cdot)$ and $h(\cdot)$.  Although we have demonstrated in the example above how to choose them, the example itself has already indicated that those choices may not be optimal.
\par
Moreover, the functions $\mu(\cdot)$ and $h(\cdot)$ are  used to set up the truncating barrier
$\mu^{-1}(h(\D))$. Once the step size is decided, the barrier is set for all states and the whole time interval.  To be more efficient, it may be worth to design a current-state-dependent truncating barrier, which then may end up with a numerical method with variable step size. We have been working on this new method and will report it later on.

\section*{Acknowledgements}
The authors would like to thank all the referees and the editor for the very useful comments and suggestions, which have helped to improve the paper a lot.
\par
The authors would like to thank
Shanghai Pujiang Program (16PJ1408000),
¡°Chenguang Program¡± supported by Shanghai Education Development Foundation and Shanghai Municipal Education Commission (16CG50),
the Natural Science Fund of Shanghai Normal University (SK201603),
Young Scholar Training Program of Shanghai's Universities,
the EPSRC (EP/K503174/1),  
the Leverhulme Trust (RF-2015-385),
the Royal Society (Wolfson Research Merit Award WM160014),
the Natural Science Foundation of China (11471216),
the Natural Science Foundation of Shanghai (14ZR1431300),
E-Institutes of Shanghai Municipal Education Commission (No. E03004)
and
the Ministry of Education (MOE) of China (MS2014DHDX020),
for their financial support.

\end{document}